\documentclass[a4paper,11pt,twoside,dvips]{amsart}
\usepackage{amssymb}
\usepackage{amsfonts}
\usepackage{amsthm}
\usepackage{amsmath}
\usepackage{newlfont}
\usepackage{graphicx}
\usepackage{graphics}
\usepackage{graphpap}
\input xy
\xyoption{all}

\newtheorem{theorem}{Theorem}
\newtheorem{lemma}{Lemma}
\newtheorem{proposition}{Proposition}
\newtheorem{remark}{Remark}
\newtheorem{property}{Property}

\newtheorem{corollary}{Corollary}

\title{Complete holomorphic vector fields on $\mathbb{C}^2$ whose underlying
foliation is polynomial}

\author{Alvaro Bustinduy}

\address{Departamento de Ingenier{\'\i}a Industrial \newline
         \indent Escuela Polit{\'e}cnica Superior \newline
         \indent Universidad Antonio de Nebrija \newline
         \indent C/ Pirineos 55, 28040 Madrid. Spain}
         \email{abustind@nebrija.es}

\thanks{Supported by MEC
projects MTM2007-61124 and  MTM2006-04785.}
\begin{document}
\begin{abstract} We extend the classification of
complete polynomial vector fields in two complex variables given by Marco Brunella to cover the case of
holomorphic (non-polynomial) vector fields whose underlying
foliation is however still polynomial.
\end{abstract}

\maketitle \markright{COMPLETE VECTOR FIELDS}

\section{Introduction and Statement of Results}
\noindent Given a holomorphic vector field $X$ on $\mathbb{C}^2$
one knows that the associated ordinary differential equation
\begin{equation*}
\dot{z}=X(z),\ z(0)=z_0\in\mathbb{C}^{2},
\end{equation*}
has a unique local solution $t\mapsto\varphi_{z}(t)$, that can be
extended by analytic continuation along paths in $\mathbb{C}$,
with origin at $t=0$, to a maximal connected Riemann surface
$\pi_z:\Omega_z(X)\to\mathbb{C}$ which is spread as a Riemann
domain over $\mathbb{C}$. The projection $\pi_z $ permits to lift
this extension  as a well-defined holomorphic function $\varphi_z:
\Omega_z (X)\to \mathbb{C}^2$ (\em see \em
\cite[p.\,126]{Forstneric}). This map  is the \em solution \em of
$X$ through $z$, and its image $C_z= \varphi_z(\Omega_z(X))$ is the
\em trajectory \em of $X$ through $z$. $X$ is \em complete \em
when $\Omega_z(X)=\mathbb{C}$ for every $z\in\mathbb{C}^2$. In
this case the flow $\varphi(t,z)=\varphi_z(t)$ of $X$ defines an
action of $(\mathbb{C},+)$ on $\mathbb{C}^{2}$ by holomorphic
automorphisms, and each trajectory of $X$, as Riemann surface
uniformized by $\mathbb{C}$ in a Stein manifold $\mathbb{C}^{2}$,
is analytically isomorphic to $\mathbb{C}$ or $\mathbb{C}^{\ast}$
(will be said of type $\mathbb{C}$ or $\mathbb{C}^{\ast}$). As an
important property, we remark that the trajectories of type
$\mathbb{C}^{\ast}$ of a complete holomorphic vector field on
$\mathbb{C}^2$  are proper (\em see \em
\cite{Suzuki-springercorto}). Let us recall that a trajectory
$C_z$ is said to be {\em proper} if its topological closure
defines an analytic curve in $\mathbb{C}^2$ of pure dimension one.
\subsection{Suzuki's Classification}
\noindent In his pioneering work \cite{Suzuki-anales} M. Suzuki classified on
$\mathbb{C}^2$: (a) complete holomorphic
vector fields whose time $t$ maps $\varphi(t,\cdot)$ of the flow
$\varphi$ are polynomials (\em algebraic flows\em), modulo
polynomial automorphism, and (b) complete holomorphic vector
fields whose trajectories are all proper (\em proper flows\em),
modulo holomorphic automorphism. The vector fields $X$ of the two
classifications together are of the forms (\em see \em \cite[Th{\'e}or{\`e}mes 2
et 4]{Suzuki-anales}):

\noindent $1)$
$$[a(x)y+b(x)]\frac{\partial}{\partial y},$$
with $a(x)$ and $b(x)$ entire functions in one variable.


\noindent $2)$
$$\lambda x \frac{\partial}{\partial x} + \mu y
\frac{\partial}{\partial y},$$ with $\lambda$, $\mu\in\mathbb{C}$.


\noindent $3)$
$$\lambda x \frac{\partial}{\partial x} + (\lambda m y + x^m)
\frac{\partial}{\partial y},$$ with $\lambda\in\mathbb{C}^{\ast}$,
$m\in\mathbb{N}$.

\noindent $4)$
$$\lambda (x^m y^n) \cdot \left \{
n x \frac{\partial}{\partial x} - m  y \frac{\partial}{\partial y}
\right \},$$ with $m,n\in\mathbb{N}^{\ast}$, $(m,n)=1$, and
$\lambda$ an entire function in  $z$ ($z=x^m y^n$).


\noindent $5)$
$$\frac{\lambda(x^m(x^{\ell}y+p(x))^n)}{x^{\ell}}\cdot \left \{
 nx^{\ell+1}\frac{\partial}{\partial x}-
 [(m+n\ell)x^{\ell} y+mp(x)+nx \dot{p} (x)]\frac{\partial}{\partial y}
 \right \},
$$
where $m,n,\ell\in\mathbb{N}^{\ast}$, $(m,n)=1$,
$p\in\mathbb{C}[x]$ of degree $< \ell$ with $p(0)\neq{0}$, and
$\lambda$ an entire function in $z$ ($z=x^m(x^{\ell}y+p(x))^n$)
with a zero of order $\geq \ell/m$ at $z=0$.

\vspace{0.15cm}

We will refer to the above list as \em Suzuki's list. \em Let us
comment some aspects of it. On case $5)$, the condition of
$\lambda$ at $z=0$ guarantees that the vector field thus defined
is holomorphic. Without this restriction, the vector field is
complete on $x\neq{0}$ but it has a pole along $x=0$. Algebraic
flows arise only for $1)$, $2)$ and $3)$, while proper flows are
defined by vector fields of $1)$, of $2)$ if
$\lambda/\mu\in\mathbb{Q}$, of $3)$ if $m=0$, and also of $4)$ and $5)$. In this
situation of proper flows, there exists always a rational first integral, given by
$x$ in the cases 1) and 3), $y^p/x^{q}$ in the case $2)$ ($p$,
$q\in\mathbb{Z}$ with $p/q=\lambda/\mu\in\mathbb{Q}$), $x^m y^n$
in the case $4)$, and $x^m(x^{\ell}y+p(x))^n$ in the case $5)$.
Therefore, modulo holomorphic automorphism, having a proper flow,
that is equivalent to the existence of a meromorphic first
integral \cite{Suzuki-springer-largo}, is equal to having a
rational first integral of one of the four types above. Flows
occurring in $2)$, with $\lambda/\mu\notin\mathbb{Q}$, and $3)$
with $m\neq{0}$, are never proper.

\subsection{Brunella's Classification}
The classification of complete polynomial vector fields on
$\mathbb{C}^2$, modulo polynomial automorphism, has been recently
obtained by M. Brunella in the outstanding work
\cite{Brunella-topology}. This classification is given by the
following vector fields (expressed in terms of Suzuki's list):

\noindent I)
$$
[cx+d]\frac{\partial}{\partial x} + Z,
$$
where $c$, $d\in \mathbb{C}$, and $Z$ is as $1)$ with $a$,
$b\in\mathbb{C}[x]$.

\noindent II)
$$
ay\frac{\partial}{\partial y} +  Z,
$$
where $a\in\mathbb{C}$, and $Z$ is as $4)$ with
$\lambda\in\mathbb{C}[z].$

\noindent III)
$$
a\left( \frac{x^{\ell}y+p(x)}{x^{\ell}} \right)
\frac{\partial}{\partial y} +  Z,
$$
where $a\in\mathbb{C}$, and $Z$ is as $5)$ with
$\lambda\in\mathbb{C}[z]$, which does not verify any
condition in the order at $z=0$, but nevertheless satisfies
the  following polynomial relation that guarantees that the sum is
holomorphic:
 $$(\ast) \quad \lambda(x^m(x^{\ell}y+p(x))^n)[mp(x)+nx \dot{p} (x)] - ap(x) \in
x^{\ell}\cdot\mathbb{C}[x,y].$$

These two classifications above are given in different contexts.
While Suzuki works with holomorphic objects (holomorphic vector
fields modulo holomorphic automorphism), Brunella is interested in
polynomial ones (polynomial vector fields modulo polynomial
automorphism). However, both are related. On one hand, each vector
field in Suzuki's list is multiple of a polynomial one by a
holomorphic function. On the other hand, each polynomial field in
Brunella's classification can be decomposed in the sum of a
complete vector field with a polynomial first integral in the form
of Suzuki's list, and a vector field which preserves this integral,
verifying moreover the necessary conditions to avoid the
rationality of the sum: $a,b\in\mathbb{C}[x]$ in I),
$\lambda\in\mathbb{C}[z]$ in II), and $\lambda\in\mathbb{C}[z]$
and satisfying ($\ast$) in III). Let us also remark that if one
does not consider these restrictions, the proofs of Propositions 1
and 2 in \cite{Brunella-topology} also work to characterize the
rational complete vector fields that preserve a polynomial of type
$\mathbb{C}$ or $\mathbb{C}^{\ast}$.
\subsection{Statement of the theorem}
The result of this work is the extension of Brunella's
classification to cover the case of non-polynomial holomorphic
vector fields whose associated foliation in $\mathbb{C}^2$ is
still polynomial, that is, defined by a polynomial vector field.
Let us observe that these vector fields admit an unique
representation of the form $f\cdot Y$, with $Y$ a polynomial
vector field with isolated singularities and $f$ a transcendental
function, up to multiplication by constants.

\begin{theorem}\label{principal1}
Let $X$ be a complete vector field on $\mathbb{C}^2$ of the form
$f \cdot Y$, where $Y$ is a polynomial vector field with isolated
singularities  and $f$ is a transcendental function. Then, all the
trajectories of $X$ are proper and, up to a holomorphic
automorphism, $X$ is in Suzuki's list.
\end{theorem}

A more precise classification, up to polynomial automorphisms,
will be stated in \S 5 and proved in sections \S2 -- \S 4.

\subsection{About the Theorem and its proof}

Let us comment some aspects of the proof. Although $Y$ is not
necessarily complete, Brunella's results \cite{Brunella-topology}
can be applied to the foliation $\mathcal{F}$ generated by $Y$ on
$\mathbb{C}^2$, extended to $\mathbb{CP}^2$. Let us first remind
some definitions. According to Seidenberg's Theorem, the minimal
resolution $\tilde{\mathcal{F}}$ of $\mathcal{F}$ is a foliation
defined on a rational surface $M$ after pulling back $\mathcal{F}$
by a birational morphism $\pi: M \to \mathbb{CP}^2$, that is a
finite composition of blowing ups. Associated to this resolution
one has: 1) The Zariski's open set $U=\pi^{-1}(\mathbb{C}^2)$ of
$M$,  over which $Y$ can be lifted to a holomorphic vector field
$\tilde{Y}$, 2) \em the exceptional divisor \em $E$ of $U$, and 3)
the \em divisor at infinity \em
$$D=M\setminus U = {\pi}^{-1}(\mathbb{CP}^2\setminus\mathbb{C}^2)=
{\pi}^{-1}(L_{\infty}),$$ that is a tree of a smooth rational
curves. The vector field $\tilde{Y}$ can be extended to $M$,
although it may have poles along one or more components of $D$.
Let us still denote this extension by $\tilde{Y}$. In
$\mathbb{C}^2$ one blows-up only singularities of the foliation,
which are in the zero set of $X$, hence $\tilde{X}$ is holomorphic
and complete on the full $U$, and its essential singularities are
contained in $D$

We start studying the cases in which $\mathcal{F}$ has rational
first integral (\S 2). The next step is the analysis of
$\mathcal{F}$ when Kodaira dimension
$\textnormal{kod}(\tilde{\mathcal{F}})$ of $\tilde{\mathcal{F}}$
is $1$ or $0$, which corresponds to the absence of a rational first
integral. The unique case in which $X$ can be determined using
directly \cite{Brunella-topology} is
$\textnormal{kod}(\tilde{\mathcal{F}})=0$ and $Y$ of type
$\mathbb{C}$ (1.- of \S 4). The remaining cases, that are
$\textnormal{kod}(\tilde{\mathcal{F}})=1$ (\S 3) and
$\textnormal{kod}(\tilde{\mathcal{F}})=0$ and $Y$ of type
$\mathbb{C}^{\ast}$ (2.- of \S 4), require to go a bit further on
\cite{Brunella-topology}. First we see that $\tilde{\mathcal{F}}$
is a Riccati foliation adapted to a fibration $g: M \to
\mathbb{P}^1$, whose projection to $\mathbb{C}^2$ by $\pi$ defines
a rational function $R$ of type $\mathbb{C}$ or
$\mathbb{C}^{\ast}$ (Lemma~\ref{lema1}). Let us denote the
$\tilde{\mathcal{F}}$-invariant components of $g$ by $\Gamma$. At
this point, one could think as a strategy to continue, that $X$ can
be determined if one proves similarly as in
\cite[Lemma\,3]{Brunella-topology} the completeness of $\tilde{X}$
on $M\setminus \Gamma$ (and thus the completeness of its
projection by $g$). Then it would be enough to see that the poles
of $\tilde{X}$ together with its essential singularities must be
contained in $\Gamma$. But this does not generally occur since $f$
can have poles and essential singularities which are transversal
to the fibers of $g$. This was pointed out to me by the referee
with Example~1.

Finally, we can avoid the previous obstacle. The principle idea is
to decompose $X$ as a complete vector field multiplied by a second
integral, implying that all its trajectories are proper
(Proposition~2). The case $R$ of type $\mathbb{C}$ is almost
direct (\S 3.1 and 1.- of \S 4). However, for the case $R$ of type
$\mathbb{C}^{\ast}$ (\S 3.2 and 2.- of \S 4) we need to prove the
presence of an invariant line by $Y$ (Lemma~\ref{lema4}). It
allows to determine $X$.

\noindent \em {\bf Example 1.} \em  Let us consider the complete
polynomial vector field
$$ Y= x \frac{\partial}{\partial x} +
\frac{\partial}{\partial y},$$ and its holomorphic first integral
$f=xe^{-y}$. The foliation $\tilde{\mathcal{F}}$ generated by $Y$
in $M=\mathbb{CP}^1\times\mathbb{CP}^1$ is Riccati with respect to
$g(x,y)=y$; it has a \em semidegenerate \em fibre over $y=\infty$,
with a saddle-node singularity at $x=\infty$, $y=\infty$, and such
that the flow of $Y$ preserves $g$.

On the other hand, $\Gamma=\{y=\infty\}$ and the complete vector
field $X=f\cdot Y$ does not project via $g$ to a complete vector
field on $g(M\setminus\Gamma)$(=$\mathbb{C}$). After a change of
coordinates $x\mapsto 1/x$, $y\mapsto 1/y$, $X$ becomes
$$
X= \frac {e^{- \frac{1}{y}}}{xy} \left \{x
\frac{\partial}{\partial x} - y^2 \frac{\partial}{\partial y}
\right \},$$ and $f$ has a first-order pole along the weak
separatrix $C=\{x=0\}$, that is transversal to $g$, and an
essential singularity along the strong separatrix $\{y=0\}\subset
\Gamma$. In this example the essential singularity is contained in
a fiber of $g$. However, by multiplying $Y$ by a transcendental
first integral $e^{xe^{-y}}$, we obtain another complete
holomorphic vector field with the essential singularity
transversal to $g$.

\vspace{0.15cm}

\noindent \em Acknowledgement. \em \newline
\noindent I want to thank the referee for
his suggestions that have improved this paper a lot. In
particular, he pointed out an error in a previous version of this
work and gave me a more simple and geometric way to prove Lemma~2.

\section{Some properties of $X$ and Rational first integrals}

\subsection{Properties}

\begin{property}\label{tipos} Types of $X$ and $Y$.
\em A complete holomorphic vector field on $\mathbb{C}^2$ is
either of type $\mathbb{C}$ or $\mathbb{C}^{\ast}$, depending on
the type of its generic (in sense of logarithmic capacity)
trajectory. Moreover in the latter case there is a meromorphic
first integral (\em see \em \cite[Th{\'e}or{\`e}me
II]{Suzuki-springer-largo}). For $Y$, which cannot be \em complete,
\em the same occurs, that if $f=0$ is empty or invariant by $Y$, the
types of $X$ and $Y$ coincide. When $f=0$ is not invariant by $Y$,
also due to the Stein-ness of $\mathbb{C}^2$, $X$ must be of type
$\mathbb{C}^{\ast}$ and $Y$ of type $\mathbb{C}$. \em
\end{property}

\begin{property}\label{segunda}
If $Y$ is complete $f$ is affine along its trajectories. \em This is
a consequence of \cite[Proposition 3.2]{Varolin}. Let us take a
point $z$ with $Y(z)\neq 0$ and the solution $\varphi_z:
\mathbb{C}\to C_z$ of $Y$ through it. The restriction of $X$ to
$C_z$, $X_{\mid C_z}$, is complete since $C_z$ outside the zeros
of $X$ is a trajectory of this vector field. As $\varphi_z$ is a
holomorphic covering map (\em see \em \cite[Proposition
1.1]{Forstneric}) $\varphi_z^{\ast}(X_{\mid C_z})$ is complete and
so affine. Therefore $$\varphi_z^{\ast}(X_{\mid
C_z})=(f\circ\varphi_z(t))\cdot\varphi_z^{\ast}(Y_{\mid C_z})=
(f\circ\varphi_z(t))\frac{\partial}{\partial t},
$$
and $(f\circ\varphi_z)(t)=at+b$, for $a$, $b\in\mathbb{C}$. In
particular $(Y f)(\varphi_z(t))=(f\circ\varphi_z)^{'}(t)$ is
constant and hence $Y^2 f =0$. Such a function $f$ is called a \em
second integral \em of $Y$. In complex geometry is important to
study these integrals. The main reason being they are the
natural tool to produce new complete vector fields. While
holomorphic first integrals of a complete $Y$ were described in
Suzuki's work, the second ones had not been extensively studied
until the recent work of D. Varolin \cite{Varolin}.  \em
\end{property}

\begin{property}\label{tercera}
If $Y$ is complete, it has a holomorphic first integral, and then
its trajectories are proper. Therefore, $X$ is in Suzuki's list.
\em There are two cases:

{\bf 1.-} \em $Y f$ is not constant \em. Then $Y f$ is an
holomorphic first integral.

{\bf 2.-} \em $Y f$ is constant. \em We observe that if $C_z$ is
of type $\mathbb{C}^{\ast}$, $f\circ\varphi_z$ is not only affine
but even constant, because $\varphi_z ^{\ast} (f \cdot Y)$ is
invariant by a group of translations, and hence $Y f=0$ along it.
$Yf=0$ on $C_z$ implies $Yf=0$ everywhere, since $Yf$ is a constant.
Hence $f$ itself is a first integral. Then we can assume that all
the trajectories of $Y$ must be of type $\mathbb{C}$. This last
property together with the fact of being $f\circ\varphi_z$ linear
implies that each trajectory $C_z$ of $Y$ (a copy of $\mathbb{C}$)
meets all the fibres of $f$ in an unique point. Then $f$ must
define a (global) fibration over $\mathbb{C}$ which is trivialized
by the trajectories of $Y$, and hence these trajectories are
proper. Moreover, according to Suzuki (\em see \em
\cite[p.\,527]{Suzuki-anales}), there is a holomorphic first
integral, which can be reduced to a coordinate after a holomorphic
automorphism.

\em
\end{property}

\subsection{Rational first integrals}

\begin{proposition}\label{firsts}
If $\mathcal{F}$ has a rational first integral, up to a polynomial
automorphism, $X$ is as \em 1), 4), \em or \em 5)  \em of Suzuki's
list. In fact, the first integral is polynomial.
\end{proposition}
\begin{proof}
Let us recall from the introduction that any complete holomorphic
vector field with a meromorphic first integral (i.e. with a proper
flow) can be transformed by a holomorphic automorphism in $1)$,
$2)$ if $\lambda/\mu\in\mathbb{Q}$, $4)$ or $5)$ of Suzuki's list,
with respectively $x$, $y^p/x^{q}$ ($p$, $q\in\mathbb{Z}$ with
$p/q=\lambda/\mu\in\mathbb{Q}$), $x^m y^n$, and
$x^m(x^{\ell}y+p(x))^n$ as first integral. More still, as $X$
has rational first integral, the reduction to one of these
possible forms can be obtained by a polynomial automorphism (\em
see \em \cite[proof of Th{\'e}or{\`e}me 4]{Suzuki-anales}). Case $2)$ is exchanged,
since $f$ is transcendental and it can not be transformed by a polynomial automorphism
in a constant map.
\end{proof}
From now on we will assume the \em absence of rational first
integrals for $\mathcal{F}$, and then for $X$. \em Thus
$\tilde{\mathcal{F}}$ admits lots of tangent entire curves; one
for each trajectory of $X$, and most of them are Zariski dense in
$M$ by Darboux's Theorem. It implies that the Kodaira dimension
$\textnormal{kod}(\tilde{\mathcal{F}})$ of $\tilde{\mathcal{F}}$
is either $0$ or $1$ \cite{Mc}. We will study these two
possibilities as in \cite[p.\,437]{Brunella-topology}.

\section{
$\textnormal{{\bf kod}}(\tilde{\mathcal{F}})=1$}

\noindent
According to McQuillan (\em see \em \cite[Section
IV]{Mc}) $\tilde{\mathcal{F}}$ must be \em a Riccati or a
Turbulent foliation, \em that is, there exists a fibration $g: M \to
B$ (maybe with singular fibres) whose generic fibre is
respectively a rational or an elliptic curve transverse to
$\tilde{\mathcal{F}}$. We will say that $g$ is adapted to
$\tilde{\mathcal{F}}$.

\begin{lemma}\label{lema1}
$\tilde{\mathcal{F}}$ is a Riccati foliation. In fact, $g_{\mid
U}$ is projected by $\pi$ as a rational function $R$ on
$\mathbb{C}^2$ of type $\mathbb{C}$ or $\mathbb{C}^{\ast}$.
\end{lemma}

\begin{proof}
It follows from Property~\ref{tipos} that $Y$ is of type
$\mathbb{C}$ or $\mathbb{C}^{\ast}$:

\em $1.-$ $Y$ of type $\mathbb{C}$. \em Then \em $\mathcal{F}$ has
only  non-dicritical singularities. \em Otherwise we had
infinitely many separatrices through a singularity, and infinitely
many of them would define algebraic trajectories by Chow's
Theorem, which would give us a rational first integral for $Y$
according to Darboux's Theorem. Therefore both \em  $E$ and $D$
are $\tilde{\mathcal{F}}$-invariant. \em But this implies that
$\tilde{\mathcal{F}}$ is a \em Riccati foliation \em because in
this situation we can always construct a rational integral for a
Turbulent $\tilde{\mathcal{F}}$ (\em see \em \cite[Lemma1]{Brunella-topology}).

On the other hand, after contracting
$\tilde{\mathcal{F}}$-invariant curves contained in fibres of $g$
(rational curves), we can assume that $g$ has no singular fibres
and that around each $\tilde{\mathcal{F}}$-invariant fibre of $g$,
$\tilde{\mathcal{F}}$ must follow one of the models described in
\cite[p.\,56]{Brunella-impa} and
\cite[p.\,439]{Brunella-topology}: nondegenerate, semidegenerate,
or nilpotent. If we now analyze \cite[the proof of Lemma
2]{Brunella-topology}, we see that it is enough to have that
$\textnormal{kod}(\tilde{\mathcal{F}})=1$, and that most of the
leaves of $\tilde{\mathcal{F}}$ are uniformized by $\mathbb{C}$,
derived from Property~\ref{tipos} to conclude that at least
one of the $\tilde{\mathcal{F}}$-invariant fibres of $g$ is
semidegenerate, or nilpotent. But this fact and the invariancy of
$E$ and $D$ by $\tilde{\mathcal{F}}$ imply that the generic fibre
of $g$ must cut $D\cup E$ in one or two points (\em see \em
\cite[the proof of Lemma 5]{Brunella-topology}). Hence the projection $R$ is
of type $\mathbb{C}$ or $\mathbb{C}^{\ast}$.

\em $2.-$ $Y$ of type $\mathbb{C}^{\ast}$. \em The leaves of
$\mathcal{F}$ are proper, and then they are properly embedded in
$\mathbb{C}^2$ \cite{Suzuki-springercorto}. As $\mathcal{F}$ has
no rational first integrals,  at least one leaf of $\mathcal{F}$
defines a planar isolated end which is properly embedded in
$\mathbb{C}^2$ and is transcendental. It follows from
\cite{Brunella-topology2} that $\mathcal{F}$ is $P$-\,complete
with $P$ a polynomial of type $\mathbb{C}^{\ast}$ or $\mathbb{C}$.
More still, as consequence of
\cite[the proof of Th\'eor\`eme]{Brunella-topology2}, $P$ is obtained as the
projection by $\pi$ of $g_{\mid U}$, that is, $R=P$ (\em see \em
also \cite [Proposition 3]{Brunella-topology}).
\end{proof}

\begin{remark}\label{propias}
\em We observe from $2.-$ of Lemma~\ref{lema1} that if the leaves
of $\mathcal{F}$ are proper $R$ is a polynomial according to
\cite[Th\'eor\`eme]{Brunella-topology2}. \em
\end{remark}

We will study the two possibilities after the previous lemma.

\subsection{$R$ of type $\mathbb{C}$}$\qquad$

\noindent By Suzuki (\em see \em \cite{Suzuki-japonesa}), up to a
polynomial automorphism, we may assume that $R=x$. Hence
$\mathcal{F}$ is a Riccati foliation adapted to $x$.  Moreover, as
the solutions of $X$ are entire maps,  they  can only avoid at
most one vertical line by Picard's Theorem. In particular $Y$ must
be of the form
$$
Cx^N\frac{\partial}{\partial x}
+[A(x)y+B(x)]\frac{\partial}{\partial y},\,\,
$$
with $C\in\mathbb{C}$, $N\in\mathbb{N}$, and $A$,
$B\in\mathbb{C}[x]$ (\em see \em also
\cite[pp.\,652-656]{Bustinduy-indiana}).

Let us take $G=f \cdot x^{N-1+\varepsilon}$ and
$F=1/x^{N-1+\varepsilon}$, with $\varepsilon=0$ if $N\geq{1}$, or
$\varepsilon=0$ or $1$ if $N=0$. Then $X$ is decomposed as the
rational complete $F\cdot Y$ of the form I) but with
$a=A/x^{N-1+\varepsilon}$ and $b=B/x^{N-1+\varepsilon}\in
1/x^{N-1+\varepsilon}\cdot\mathbb{C}[x]$, where $d=0$ and $c=C$,
if $N\geq{1}$ or $N=\varepsilon=0$, or $c=0$ and $d=C$, if $N=0$
and $\varepsilon=1$, multiplied by $G$. We observe that
$dR(F\cdot Y)= c. R$ or $d$, and we conclude that $X$ has the form
$i)$ of Theorem~\ref{principal}.

\subsection{$R$ of type $\mathbb{C}^{\ast}$}$\qquad$

\noindent By Suzuki (\em see \em \cite{Suzuki-anales}), up to a
polynomial automorphism, we may assume that $R=x^
{m}(x^{\ell}y+p(x))^{n}$, where $m\in\mathbb{N}^\ast$, $n
\in\mathbb{Z}^{\ast}$, with $(m,n)=1$, $\ell\in\mathbb{N}$,
$p\in\mathbb{C}[x]$ of degree $< \ell$ with $p(0)\neq{0}$ if
$\ell>0$ or $p(x)\equiv{0}$ if $\ell=0$.

\subsubsection*{New coordinates} According to relations
$ x=u^n\,\,\,\,\textnormal{and}\,\,\,\,x^{\ell}y+p(x)=v \,u^{-m},$
it is enough to take the rational map $H$ from $u\neq {0}$ to
$x\neq{0}$ defined by
\begin{equation}\label{relaciones}
(u,v)\mapsto (x,y)=(u^n, {u^{-(m+n\ell)}} [v-u^m p(u^n)])
\end{equation}
in order to get $R\circ H(u,v)=v^n$.

Although $R$ is not necessarily a polynomial ($n\in\mathbb{Z}$),
it follows from the proof of \cite[Proposition 3.2]
{Bustinduy-indiana} that $H^{\ast}\mathcal{F}$ is a Riccati
foliation adapted to $v^n$ having $u=0$ as invariant line. Thus

\begin{equation} \label{hest}
\begin{split}
H^{\ast} X = & (f \circ H)\cdot H^{\ast} Y \\
= & (f \circ H(u,v))\cdot u^{k}\cdot Z \\
= & (f \circ H(u,v))\cdot u^{k} \cdot \left\{a(v)u
\frac{\partial}{\partial u} + c(v)\frac{\partial}{\partial
v}\right\},
\end{split}
\end{equation}
where $k\in\mathbb{Z}$, and $a$, $c\in\mathbb{C}[v]$

Our goal now is to prove that in (\ref{hest}) the polynomial
$c(v)$ is a monomial $cv^N$. It will be a consequence of the
following lemma.

\begin{lemma}\label{lema4}
The line $x=0$ is invariant by $Y$.
\end{lemma}
\begin{proof}
Take the Riccati foliation $\tilde{\mathcal{F}}$ on $M$, and let
$F$ be the fibre over $0$. It follows from the local study of
\cite{Brunella-topology} or \cite{Brunella-impa} that \em at most
\em one irreducible component of $F$ can be non-invariant by
$\tilde{\mathcal{F}}$ (just look at the blow-up of models of
\cite{Brunella-topology}). Moreover, if such a non-invariant
component exists, then it is everywhere transverse to the
foliation. This settles immediately the case $\ell=0$ in $R$ since
at least one irreducible component of $\{xy=0\}$ must be
invariant.

In the case $\ell>0$, $\{R=0\}$ has two disjoint components, one
(the axis  $\{x=0\}$) isomorphic to $\mathbb{C}$ and another
isomorphic to $\mathbb{C}^{\ast}$. We want to prove that the first
is necessarily invariant. Let us assume the contrary. Let $C$ be
the irreducible component of $F$ corresponding to $\{x=0\}$ and
assume that it is transverse to $\tilde{\mathcal{F}}$. There is
one and only one point $p\in C$ which belongs to the divisor at
infinity $D$. This point is also the unique intersection point
between $C$ an the other components of $F$. Because $D$ and $F
\setminus C$ are invariant, and the foliation is regular at $p$,
we see that there exists a common irreducible component $E\subset
D\cap F$ such that, on a neighborhood $U$ of $C$, we have
$$
D\cap U = E \cap U\quad \textnormal{and} \quad F\cap U=(E\cap
U)\cup C.
$$

Now, by contracting components of $F$ different from $C$, we get a
model $C_{0}$ like (a) of \cite{Brunella-topology} (not like (b),
which contains two quotient singularities). The direct image $
D_0$ of $D$ is then an invariant divisor which cuts $C_0$ at a
single point $p_0$. Hence it cuts a generic fibre also at a single
point, which contradicts that $R$ is of type $\mathbb{C}^{\ast}$.
\end{proof}

By Lemma~\ref{lema4}, as $H$ is a finite covering map from $u\neq
0$ to $x\neq 0$, $H^{\ast}X$ is complete on $u\neq 0$. Thus
according to Picard's Theorem its solutions are entire maps which
can avoid at most one horizontal line, and hence $c(v)$ in
(\ref{hest}) is of the form $cv^N$ with $c\in\mathbb{C}$,
$N\in\mathbb{N}$.

We can write $H^{\ast} X$ as the product of the complete field
$1/v^{N-1+\varepsilon}\cdot Z$ in $u\neq{0}$ by the function $f
\circ H(u,v)\cdot u^{k}\cdot v^{N-1+\varepsilon}$, where
$\varepsilon=0$ if $N\geq{1}$, or $\varepsilon=0$ or $1$ if $N=0$.

\begin{proposition}
$Y$ has proper trajectories
\end{proposition}
\begin{proof}
If $X$ is of type $\mathbb{C}^{\ast}$ it follows by
\cite{Suzuki-springercorto}. If $X$ is of type $\mathbb{C}$, with
the notations of \S $3.1$ an \S $3.2$, we distinguish two cases:

First case: \em R of type \em $\mathbb{C}$. Assume that $F\cdot Y$
is of type $\mathbb{C}$ by \cite{Suzuki-springercorto}. As $F\cdot
Y$ is complete the restriction of $G$  to each solution
$\varphi_z$ of that field is constant (Property~\ref{segunda}),
and hence $G$ is a meromorphic first integral of $Y$.

Second case: \em R of type \em $\mathbb{C}^{\ast}$. Assume that
$1/v^{N-1+\varepsilon}\cdot Z$ is of type $\mathbb{C}$ by
\cite{Suzuki-springercorto}. One sees that $(f \circ H(u,v)\cdot
u^{k}\cdot v^{N-1+\varepsilon})\circ \varphi_z$ must be constant
for each entire solution $\varphi_z$ of that field through points
$z$ in $u\neq{0}$ (Property~\ref{segunda}). Then, according to
(\ref{relaciones}), ${(f \circ H(u,v)\cdot u^{k}\cdot
v^{N-1+\varepsilon})}^{mn}$ is projected by $H$ as
$$
f^{mn}\cdot x^{mk}\cdot {(x^{m}{(x^{\ell} y +
p(x))}^{n})}^{m(N-1+\varepsilon)}
$$
thus obtaining a meromorphic first integral of $Y$.

\end{proof}

\noindent \em The global one form of times. \em Let us take the
one-form $\eta$ obtained when we remove the codimension one zeros
and poles of $dR(x,y)$.  The contraction of $\eta$ by $Y$,
$\eta(Y)$, is a polynomial, which vanishes only on components of
fibres of $R$, since $Y$ has only isolated singularities. In fact,
the number of these fibres over nonzero values is at most one.
Otherwise the entire solutions of $X$ would be projected by $R$,
avoiding at least two points, which is impossible by Picard's
Theorem. Then, \em up to multiplication by constants: \em
\begin{equation}\label{polinomio}
\eta(Y)=
\begin{cases}
x^{\alpha}\cdot{(x^{\ell} y + p(x))}^{\beta} \cdot
{(x^{m}{(x^{\ell} y + p(x))}^{n} - s )}^{\gamma},\,\,\text{if}\,\,n>0;\\
x^{\alpha}\cdot{(x^{\ell} y + p(x))}^{\beta}\cdot {(
x^{m}-s{(x^{\ell} y + p(x))}^{-n})}^{\gamma},\,\,\text{if}
\,\,n<0.
\end{cases}
\end{equation}
where $\alpha$, $\beta$, $\gamma\in \mathbb{N}$, and
$s\in\mathbb{C}^{\ast}$.

Let us define $\tau=[1/(f\cdot\eta(Y))]\cdot \eta$. This one-form
on $\{f\cdot\eta(Y)\neq{0}\}$ coincides locally along each
trajectory of $X$  with the \em differential of times \em given by
its complex flow. It is called the global one-form of times for
$X$. Moreover $\tau$ can be easily calculated attending to
(\ref{polinomio}) as
\begin{equation}\label{tau}
\tau= \dfrac{x(x^{\ell} y + p(x))}{f \cdot \eta(Y)} \cdot \dfrac{d
R}{R}.
\end{equation}
In $(u,v)$ coordinates  we then get

\begin{equation}\label{rho1}
\varrho=H^{\ast}\tau=
\begin{cases}
\dfrac{u^{m(\beta-1)-n(\alpha-1)}}{(f\circ H) \cdot
v^{\beta-1}\cdot {(v^n-s)}^{\gamma}}\,\cdot \dfrac{d v^n}{v^n},
\,\,\text{if}
\,\,n>0;\\
\\
\dfrac{u^{m(\beta-1)-n(\alpha-1) -mn\cdot\gamma}}{(f\circ H) \cdot
v^{\beta-1}\cdot {(1-s v^{-n})}^{\gamma}}\,\cdot\dfrac{d
v^n}{v^n},\,\,\text{if} \,\,n<0.
\end{cases}
\end{equation}

It holds that $\varrho(H^{\ast} X)\equiv{1}$. Since $\varrho -
1/[(f \circ H(u,v) \cdot u^k \cdot c v^N)]\,dv$ contracted by
$H^{\ast}X$ is identically zero  and we are assuming that there is
no rational first integral,
\begin{equation}\label{rho2}
\varrho=1/[(f \circ H(u,v) \cdot u^k \cdot c v^N)]\,dv.
\end{equation}
Therefore (\ref{rho1}) and (\ref{rho2}) must be equal and $k$ of
(\ref{hest}) can be explicitly calculated. Finally, let us observe
that  for any path $\epsilon$ contained in a trajectory of $X$
from $p$ to $q$, that can be lifted by $H$ as $\tilde{\epsilon}$,
$\int_{\tilde{\epsilon}}\,\varrho$, represents the complex time
required by the flow of $X$ to travel from $p$ to $q$.

We may assume that $\gamma=0$, $\beta=N$ and $\alpha>0$ in
(\ref{rho1}). Moreover according to Remark~\ref{propias} we can
also assume that \em $R$ is a polynomial and that $n>0$. \em

Let us observe that $Y$ can be explicitly calculated as
\begin{multline}\label{Y}
Y=u^{k}\cdot H_{\ast} (a(v)u\frac{\partial}{\partial u} +cv^N
\frac{\partial}{\partial v})=\\
\\
=u^{k}\cdot \left(
\begin{array}{cc}
nu^{n-1} & 0 \\
      \\
\dfrac{n\ell u^mp(u^n)-u^{n+m}p'(u^n)-(m + n\ell)v}{u^{m+n\ell+1}}
& \dfrac{1}{u^{m+n\ell}}
\end{array}
\right)\cdot \left(
\begin{array}{c}
a(v)u\\
                                     \\
cv^N
\end{array}
\right)
\end{multline}
where $u={x}^{1/n}$ and $v={x}^{m/n}\,(x^{\ell}y + p(x))$.

We analyze two cases:

\noindent $\bullet$ $N\geq {1}$. We show that \em each term
$1/v^{N-1}\cdot Z$ and $f \circ H(u,v)\cdot u^{k}\cdot v^{N-1}$ of
the decomposition of $H^{\ast}X$  can be separately projected by
$H$. \em Let us observe that $a(0)\neq{0}$. Otherwise $Y$ had not isolated
singularities since $N>0$. The first component
$$n {x}^{(k+n)/n}a({x}^{m/n}\,(x^{\ell}y + p(x)))$$
of (\ref{Y}) must be a polynomial. Since $k=n(\alpha-1)-m(N-1)$ by
(\ref{rho1}) and (\ref{rho2}), $k=n\cdot\delta$ with
$\delta\in\mathbb{Z}$. On the other hand $(m,n)=1$, and it implies
that $N-1=n \cdot \kappa$ with $\kappa\in\mathbb{Z}$. Using
(\ref{relaciones}), one gets

\begin{equation}\label{FG}
\begin{split}
H_{\ast}(f \circ H(u,v)\cdot u^k \cdot v^{N-1})= &G= f\cdot
x^{\delta}\cdot{(x^{m}{(x^{\ell} y +
p(x))}^{n})}^{\kappa}\\
H_{\ast}(1/v^{N-1}\cdot Z)= & F\cdot Y={1/(x^{m}{(x^{\ell} y +
p(x))}^{n})}^{\kappa} \cdot Y.
\end{split}
\end{equation}
Finally, as $dv^n(1/v^{N-1}\cdot Z)=nc \cdot v^n $, $dR (F\cdot
Y)=nc \cdot R$. If one now defines $G$ and $F$ according to
(\ref{FG}), and $\Omega=nc$, $X$ is as in $ii)$ and $iii)$ of $A)$
in Theorem~\ref{principal}.

\noindent $\bullet$ $N=0$. As $Y$ is a polynomial vector field
with isolated singularities, a simple inspection of the two
components in (\ref{Y}) implies that $k=m+n\ell$ and
$a\in(1/z)\cdot \mathbb{C}[z^n]$, with $a(0)=0$ if $n>1$. Finally,
according to (\ref{Y}), one sees that $1/x^{(m + n \ell)/n} \cdot
Y$ is obtained as the projection of a complete vector field
whose trajectories are of type $\mathbb{C}$. Therefore $X$ is as
in B) of Theorem~\ref{principal}.

This finishes the part of
$\textnormal{kod}(\tilde{\mathcal{F}})=1$.

\section{
$\textnormal{{\bf kod}}(\tilde{\mathcal{F}})=0$}

\noindent \em  In what follows we may suppose that $Y$ is of type
$\mathbb{C}$. \em If $Y$ is of type $\mathbb{C}^{\ast}$, as $f=0$
is empty or invariant by $Y$, $X$ is also of type
$\mathbb{C}^{\ast}$ (Property~\ref{tipos}). According to
\cite{Suzuki-springercorto} the leaves of type $\mathbb{C}^{\ast}$
of $\mathcal{F}$ are proper and they are properly embedded in
$\mathbb{C}^2$. Moreover, if these leaves of $\mathcal{F}$ are
algebraic there exists a rational integral by Darboux's Theorem.
Therefore at least one leaf of type $\mathbb{C}^{\ast}$ of
$\mathcal{F}$ defines a planar isolated end which is properly
embedded in $\mathbb{C}^2$ and is transcendental, and then
$\mathcal{F}$ is $P$-\,complete with $P$ a polynomial of type
$\mathbb{C}^{\ast}$ or $\mathbb{C}$ \cite [Proposition
3]{Brunella-topology} (\em see \em also 2.- of Lemma~\ref{lema1}).
This is enough to apply the results of \S 3.1 and \S 3.2.

According to \cite[Section IV]{Mc} we can contract
$\tilde{\mathcal{F}}$-invariant rational curves on $M$ (via the
contraction $s$) to obtain a new surface $\bar{M}$ (maybe
singular), a reduced foliation $\bar{\mathcal{F}}$ on this
surface, and a finite covering map $r$ from a smooth $S$ to
$\bar{M}$ such that: 1) $r$ ramifies only over (quotient)
singularities of $\bar{M}$ and 2) the foliation
$r^{\ast}(\bar{\mathcal{F}})$ is generated by a holomorphic vector
field $Z_{0}$ on $S$ with isolated zeroes. It follows from \cite
[p.\,443]{Brunella-topology} that the covering $r$ can be lifted
to $M$ via a birational morphism $g:T \to S$ and a ramified
covering $h:T \to M$ such that $s \circ h = r \circ g$. So we have
the following diagram:
$$
\xymatrix{  M\ar[d]_{s}
& \ar[l]_{h} T \ar[dl]_{s \circ h}^{r \circ g} \ar[d]^{g} \\
  \bar{M} & \ar[l]^{r} S}
$$
This construction guarantees the existence of two open sets $V, W
\subset T$ with the property that the covering $\pi \circ h: V \to
\mathbb{C}^2\setminus \pi(E)$ is either unramified ($V=W$) or it
ramifies only over a line $L$ in $\mathbb{C}^2\setminus \pi(E)$
($V\neq W$). It allows to lift $Z_0$ via $g$ as a rational vector
field $Z$ on $T$ generating
$g^{\ast}(r^{\ast}(\bar{\mathcal{F}}))=h^{\ast}(\tilde{\mathcal{F}})$,
verifying that it is holomorphic and complete on $W$, with
a pole along $V \setminus W$ \cite [Lemma 7]{Brunella-topology}.
One analyzes the two possibilities above:

$1.-$ If $V = W$, using the regular cover $\pi \circ h: V \to
\mathbb{C}^2\setminus \pi(E)$ that is trivial, one can extend $Z$
to a finite set of points obtaining a complete polynomial
vector field that generates $\mathcal{F}$. Therefore $Y=P\cdot Z$
with $P$ a polynomial, which must be constant since $Y$ has
isolated singularities. So we can assume that $Y=Z$ and then $f$
is a holomorphic second integral of $Y$. According to
Property~\ref{tercera}, $Y$ has a holomorphic first integral, and
hence its trajectories are proper.

On the other hand, we know that the flow of $Y$ is algebraic,
since this vector field arises from $Z_0$ on $S$ which generates
algebraic automorphisms of $S$. As a consequence, after a
polynomial automorphism, $Y$ has to be one of these two vector
fields (\em see \em \cite {Brunella-topology}):
$$a) \,\,\,  \lambda x \frac{\partial}{\partial x} + \mu y
\frac{\partial}{\partial y}, \quad \lambda,\mu\in\mathbb{C},\,
\lambda/\mu\notin \mathbb{Q},
$$
$$ \quad \quad \,b)\,\,
\,\lambda x \frac{\partial}{\partial x} + (\lambda m y + x^m)
\frac{\partial}{\partial y},\quad \lambda\in\mathbb{C},\,
m\in\mathbb{N}.
$$
\noindent Cases $a)$ and $b)$ with $m>0$ never have proper
trajectories. Therefore $Y$ is as $b)$ with $m=0$, and $X$ has the
form $i)$ of Theorem~\ref{principal}.

\begin{remark}\label{obs2}
\em Once $Y$ have been determined, $f$ can be easily obtained. As
$Yy=1$, $f=H \cdot y + G$, with $H$ and $G$ first integrals of
$Y$. It is enough to define $H=Yf$ and $G=f-y\cdot Yf$. On the
other hand, computing the flow of $Y$, we can see directly that
its trajectories are contained in the level sets of $xe^{-\lambda
y}$. Finally, according to Stein Factorization Theorem,
$H=h(xe^{-\lambda y})$ and $G=g(xe^{-\lambda y})$ where $h$ and $g$ are
entire functions in one variable. \em
\end{remark}

$2.-$ If $V \neq  W$, then $Y$ is of type $\mathbb{C}^{\ast}$
\cite[p.\,445]{Brunella-topology} which contradicts our
assumptions.

\section{
Polynomial version of Theorem~\ref{principal1}}

\begin{theorem}\label{principal}
Let $X$ be a complete vector field on $\mathbb{C}^2$ of the form
$f \cdot Y$, where $Y$ is a polynomial vector field with isolated
singularities and $f$ is a transcendental function. Then, all the
trajectories of $X$ are proper and, up to a polynomial
automorphism, $X$ can be decomposed as $G \cdot F\cdot Y$ in one
of the two following cases:

\noindent {\bf \em A)\em} $G$ is a meromorphic function that is
affine along the trajectories of a rational complete vector field
$F \cdot Y$ such that $dR (F\cdot Y)=\Omega \cdot R^{j}$, where
$\Omega\in\mathbb{C}$, $j=0$ or $1$, $R$ is a polynomial of type
$\mathbb{C}$ or $\mathbb{C}^{\ast}$, and $F$ is constant or equal to $x^{-\delta}$, with $\delta\in \mathbb{Z}$ along the
fibres of $R$. Explicitly, $X$ is defined by the following forms:

\noindent {\bf i)} The case $R=x$, where

\vspace{-0.15cm}
\begin{enumerate}

\item[{\em -\em}] $F\cdot Y$ is as in \em I); \em

\item [{\em -\em}] $a,\,b\in
1/x^{N-1+\varepsilon}\cdot\mathbb{C}[x]$, where $d=0$ if
either with $N\geq{1}$ or $N=\varepsilon=0$, and $c=0$ if $N=0$,
$\varepsilon=1$;

\item[{\em -\em}] $F=1/x^{N-1+\varepsilon}$, $G=f \cdot
x^{N-1+\varepsilon}$ where $N\in\mathbb{N}$, and where $\varepsilon=0$ if $N\geq{1}$, or
$\varepsilon=0,1$ if $N=0$.

\end{enumerate}

\noindent {\bf ii)}The case $R=x^my^n$, where

\vspace{-0.15cm}
\begin{enumerate}

\item [{\em -\em}] $F\cdot Y$ is as in \em II); \em

\item [{\em -\em}] $\lambda\in 1/z^{\kappa}\cdot\mathbb{C}[z]$,
$z=x^my^n$;

\item[{\em -\em}] $F=1/({(x^{m}y^{n})}^{\kappa}\cdot x^{\delta})$, $G=f \cdot
x^{\delta}\cdot {(x^{m}y^{n})}^{\kappa}$; with $\kappa,\delta\in\mathbb{Z}$, $m,n\in\mathbb{N}^{\ast}$ and
$(m,n)$ $=1$.

\end{enumerate}

\noindent {\bf iii)} The case $R=x^m(x^{\ell}y+p(x))^n$, where

\vspace{-0.15cm}
\begin{enumerate}

\item [{\em -\em}] $F\cdot Y$ is as in \em III); \em

\item [{\em -\em}] $\lambda\in 1/z^{\kappa}\cdot\mathbb{C}[z]$,
$z=x^m(x^{\ell}y+p(x))^n$;

\item[{\em -\em}] $F=1/({(x^m(x^{\ell}y+p(x))^n)}^{\kappa}\cdot x^{\delta})$, $G=f
\cdot x^{\delta}\cdot {(x^m(x^{\ell}y+p(x))^n)}^{\kappa}$; with $\kappa,\delta\in\mathbb{Z}$, $m,n,\ell\in\mathbb{N}^{\ast}$,
$(m,n)=1$, $p\in\mathbb{C}[x]$ of degree $< \ell$, and  $p(0)\neq{0}$.

\end{enumerate}

\vspace{0.25cm}

\noindent {\bf \em B)\em} $G=f\cdot x^{(m + n \ell)/n}$ is a
multivaluated holomorphic function that is affine along the
trajectories of $F\cdot Y$, which is a multivaluated complete
vector field with all its trayectories of type $\mathbb{C}$
defined by the product of $F=1/x^{(m + n \ell )/n}$ by the
polynomial vector field
$$
Y=u^{m+n(\ell+1)} a(v)\frac{\partial}{\partial x} +  [n\ell
u^mp(u^n)-u^{n+m}p'(u^n)-(m + n\ell)v]a(v) + c
\frac{\partial}{\partial y}
$$
where $u={x}^{1/n}$, $v={x}^{m/n}\,(x^{\ell}y + p(x))$, with
$m,n\in\mathbb{N}^{\ast}$, $(m,n)=1$, $\ell\in\mathbb{N}$,
$p\in\mathbb{C}[x]$ of degree $< \ell$, $p(0)\neq{0}$ if $\ell>0$
or $p(x)\equiv{0}$ if $\ell=0$, $c\in\mathbb{C}^{\ast}$, and $a\in
(1/z)\cdot \mathbb{C}[z^n]$, with $a(0)=0$ if $n>1$.

\end{theorem}

\begin{remark}
\em Vector fields in A) are obtained by multiplication of a
rational complete one $F \cdot Y$ in (rational) Brunella's
classification and a \em meromorphic \em second integral $G$, that
is, a function which is affine along the trajectories of $F \cdot
Y$. It is important to remark that in Brunella's list there are
non-proper vector fields, which do not appear in our classification
due to the existence of $f$.

On the other hand, any $X$ of $ii)$ and $iii)$ can be expressed
after the rational change of coordinates $H$ given by
(\ref{relaciones}) as
$$
H^{\ast} X= f\circ H(u,v) \cdot  u^{n \delta} \cdot \left\{
a(v)u\frac{\partial}{\partial u} +cv^N \frac{\partial}{\partial
v}\right\},
$$

Let us also note that $1)$, $4)$ and $5)$ of Suzuki's list define
respectively cases $i)$, $ii)$ and $iii)$ with polynomial first
integral.  However, A) contains other different vector fields
since $i)$, $ii)$ and $iii)$ can not be reduced in general to
$1)$, $4)$ and $5)$ by a polynomial automorphism.

Vector fields $X$ in B) can be expressed after  $H$ as
$$
H^{\ast} X= f\circ H(u,v) \cdot u^{m+n\ell}\cdot
\left\{a(v)u\frac{\partial}{\partial u} +c
\frac{\partial}{\partial v}\right\}.
$$
The following Example\,2 gives us one $X$ in B) with an explicit $f$ which is
not in A).

\em
\end{remark}

\noindent \em {\bf Example 2.} \em Let us consider

$$
X= f \cdot Y= e^{-(m/nc)\cdot x^m y^n} \left \{x^{1+m}y^{n-1}
\frac{\partial}{\partial x} - (m x^{m} y^{n} - c)
\frac{\partial}{\partial y} \right \}.
$$

We see that $X$ is as in B) of Theorem~\ref{principal}, $Y$ is as
in B) with $\ell=0$ and $a(z)=(1/z)\cdot z^{n}=z^{n-1}$, and $X$ is
complete, since according to (\ref{Y}), $H^{\ast} X$ equals to
$$
e^{-(m/nc)\cdot v^n}\cdot H^{\ast} Y={(u \cdot e^{-(1/nc)\cdot
v^n})}^{m}\cdot \left\{v^{n-1}u \frac{\partial}{\partial u} +
c\frac{\partial}{\partial v}\right\}
$$
is a complete polynomial vector field multiplied by a first
integral.

\bibliographystyle{plain}

\begin{thebibliography}{10}


\bibitem{Brunella-topology}
M.~Brunella.
\newblock Complete vector fields on the complex plane.
\newblock {\em Topology} {\bf 43}(2): 433--445, 2004.



\bibitem{Brunella-impa}
M.~Brunella.
\newblock Birational geometry of foliations.
\newblock {\em First Latin American Congress of Mathematicians,} IMPA,
2000.



\bibitem{Brunella-topology2}
M.~Brunella.
\newblock Sur les courbes int{\'e}grales propres des champs de vecteurs
polynomiaux.
\newblock {\em Topology} {\bf 37}(6): 1229--1246, 1998.


\bibitem{Bustinduy-indiana}
A.~Bustinduy.
\newblock On the entire solutions of a polynomial vector field on
$\mathbb{C}^2$.
\newblock
{\em Indiana Univ. Math. J.} {\bf 53} (2004), 647--666.





\bibitem{Forstneric}
F.~Forstneric.
\newblock Actions of $(\mathbf{R}^2,+)$ and $(\mathbf{C}^2,+)$ on
complex manifolds
\newblock {\em Math. Z.} {\bf 223}: 123--153, 1996.





\bibitem{Martinet-Ramis}
J.~Martinet et J.-P. Ramis.
\newblock
Probl{\`e}mes de modules pour des {\'e}quations
diff{\'e}rentielles non lin{\'e}aires du premier ordre.
\newblock
{\em Inst. Hautes {\'E}tudes Sci. Publ. Math.} {\bf 55} (1982),
63--164.


\bibitem{Mc}
M.~McQuillan.
\newblock
Non-commutative Mori Theory.
\newblock {\em Preprint IHES}, M/01/42, 2001.



\bibitem{Scardua}
B.~A. Scardua.
\newblock A remark on parabolic projective foliations.
\newblock {\em Hokkaido Mathematical Journal}  {\bf 28} (1999),
231--252.


\bibitem{Suzuki-japonesa}
M.~Suzuki.
\newblock Propri\'et\'es topologiques des polyn\^omes de deux variables
complexes, et automorphismes alg\'ebriques de l'espace
$\mathbb{C}^{2}$.
\newblock {\em Ann. Sci. {\'E}cole Norm. Sup. (4)} {\bf 10}(4):
517--546, 1977.



\bibitem{Suzuki-anales}
M.~Suzuki.
\newblock Sur les op{\'e}rations holomorphes du groupe additif
complexe sur l'espace de deux variables complexes.
\newblock {\em J. Math. Soc. Japan } {\bf 26}:
241--257, 1974.

\bibitem{Suzuki-springercorto}
M.~Suzuki.
\newblock Sur les op\'erations holomorphes de $\mathbb{C}$\ et de
$\mathbb{C}^{\ast}$\ sur un espace de {S}tein.
\newblock {\em Lecture Notes in Math.} {\bf 670}: 80--88, 1978.


\bibitem{Suzuki-springer-largo}
M.~Suzuki.
\newblock Sur les int{\'e}grales premi{\`e}res de certains feuilletages
analytiques complexes.
\newblock {\em Lecture Notes in Math.} {\bf 670}: 53--79, 1978.

\bibitem{Varolin}
D.~Varolin.
\newblock A general Notion of Shears, and Applications.
\newblock {\em Michigan Math. J.} {\bf 46}: 533--553, 1999.




\end{thebibliography}
\def\cprime{$'$} \def\cprime{$'$}

\end{document}